\documentclass{lms}

\usepackage{amsmath,amssymb}

\newtheorem{theorem}{Theorem}[section]

\newtheorem{corollary}[theorem]{Corollary}
\newtheorem{lemma}[theorem]{Lemma}

\numberwithin{equation}{section}

\newcommand{\R}{\mathbb{R}}
\newcommand{\N}{\mathbb{N}}

\newcommand{\C}{\mathbb{C}}

\newcommand{\ov}{\overline}
\newcommand{\dis}{\displaystyle}
\newcommand{\supp}{\textup{supp}}

\newcommand{\Notequiv}{/\kern-.6em\hbox{$\equiv$} }

\title[Green potentials and Blaschke products]%
    {Inequalities for sums of Green potentials and Blaschke products}
\author{Igor E. Pritsker}%

\classno{31A15 (primary), 30C15 (secondary)}
\extraline{Research was partially supported by the National Security Agency, and by the Alexander von Humboldt Foundation.}

\begin{document}


\maketitle



\begin{abstract}

We study inequalities for the infima of Green potentials on a compact subset of an arbitrary domain in the complex plane. The results are based on a new representation of the pseudohyperbolic farthest-point distance function via a Green potential. We also give applications to sharp inequalities for the supremum norms of Blaschke products.

\end{abstract}


\section{Green potentials} \label{sec1}

Let $G\subset\ov{\C}$ be a domain possessing the Green function $g_G(z,\zeta)$ with pole at $\zeta\in G$. For the positive Borel measures $\nu_k, k=1,\ldots,m,$ with compact supports in $G$, define their Green potentials \cite[p. 96]{AG} by
\[
U_G^{\nu_k}(z):=\int g_G(z,\zeta)\,d\nu_k(\zeta),\quad z\in G.
\]
Note that Green potentials are superharmonic and nonnegative in $G$.
Suppose that $\nu:=\sum_{j=1}^m \nu_j$ is a unit measure. We study inequalities of the following type
\begin{align} \label{1.1}
\sum_{k=1}^m \inf_E U_G^{\nu_k} \ge A + B\, \inf_E \sum_{k=1}^m U_G^{\nu_k},
\end{align}
where $E\subset G$ is a compact set of positive logarithmic capacity ${\rm cap}(E)$. Observe that the opposite inequality is always true with $A=0$ and $B=1$:
\[
\sum_{k=1}^m \inf_E U_G^{\nu_k} \le \inf_E \sum_{k=1}^m U_G^{\nu_k},
\]
which suggests that optimal constants in \eqref{1.1} should satisfy $A\le 0$ and $B\le 1.$ This problem was considered for logarithmic potentials in \cite{PS}, where an analog of \eqref{1.1} was proved with constants $B=1$ and $A$ expressed through the farthest-point distance function and potential theoretic quantities. We generalize those ideas to Green potentials and the hyperbolic setting below. Furthermore, we prove a sharp version of \eqref{1.1}, and provide applications to the inequalities for Blaschke products in supremum norms on $E$.

Since $E$ has positive capacity by our assumption, we can consider its Green equilibrium measure $\mu_E^G$ and the minimum Green energy (Robin's constant) $V_E^G$, see \cite[p. 132]{ST}. It is well known that $\mu_E^G$ is a positive unit Borel measure supported on $\partial E$, whose potential satisfies
\begin{align} \label{1.2}
U_G^{\mu_E^G}(z) \le V_E^G,\ z\in G, \quad \mbox{and} \quad U_G^{\mu_E^G}(z) = V_E^G\ \mbox{q.e. on } E,
\end{align}
see Theorem 5.11 \cite[p. 132]{ST}. The equality in \eqref{1.2} holds quasi everywhere (q.e.) on $E$, which means that it holds up to an exceptional set of zero logarithmic capacity. Hence Fubini's theorem gives that
\[
\int U_G^{\mu}(z)\,d\mu_E^G(z) = \int U_G^{\mu_E^G}(z)\,d\mu(z) \le V_E^G
\]
for any positive unit Borel measure $\mu$ supported in $G.$ This immediately implies that
\begin{equation} \label{1.3}
\inf_E U_G^{\mu} \le V_E^G,
\end{equation}
and that for $\nu=\sum_{j=1}^m \nu_j$ we have
\[
\sum_{k=1}^m \inf_E U_G^{\nu_k} \ge 0 \ge -V_E^G + \inf_E \sum_{k=1}^m U_G^{\nu_k}.
\]
While this basic version of \eqref{1.1} may be useful, it is clearly not sharp as equality is never attained in the above inequality. A sharp version of \eqref{1.1} requires more sophisticated tools.

For $z,\zeta$ in the unit disk $D$, consider the pseudohyperbolic metric $\delta_D(z,\zeta)=|z-\zeta|/|1-\bar\zeta z|.$ It has a standard extension to an arbitrary simply connected domain by conformal invariance. However, we make an even more general definition of the pseudohyperbolic metric for any domain $G\subset\ov{\C}$ possessing the Green function $g_G(z,\zeta)$ with pole at $\zeta\in G$:
\[
\delta_G(z,\zeta):= e^{-g_G(z,\zeta)},\quad z,\zeta \in G.
\]
Thus the function $-\log \delta_G(z,\zeta) = g_G(z,\zeta)$ is superharmonic in each variable $z,\zeta \in G$. If we define the farthest-point pseudohyperbolic distance function for a compact set $E\subset G$ by
\[
d_E^G(z) := \sup_{\zeta\in E} \delta_G(z,\zeta),
\]
then the function $-\log d_E^G(z) = \inf_{\zeta\in E} g_G(z,\zeta)$ is superharmonic in $G$, see Theorem 2.4.7 of \cite[p. 38]{Ra}. It is clear from Harnack's inequality applied to the Green function that $-\log d_E^G(z)$ is continuous in $G$, provided $E$ is not a singleton.  Furthermore, we have the following representation (Riesz decomposition) of this superharmonic function as a Green potential.

\begin{theorem} \label{thm1.1}
Let $G\subset\ov{\C}$ be a domain with the Green function $g_G(z,\zeta)$, and let $E \subset G$ be a compact set, which is not a single point. Then
\begin{align} \label{1.4}
-\log d_E^G(z) = \inf_{\zeta\in E} g_G(z,\zeta) = \int g_G(z,\zeta)\, d\sigma_E^G(\zeta), \quad z\in G,
\end{align}
where $\sigma_E^G$ is a positive measure supported on $G$, whose total mass satisfies
\begin{align} \label{1.5}
\sigma_E^G(G)\le 1.
\end{align}

If $G$ is simply connected, then strict inequality holds in \eqref{1.5}, and the support of $\sigma_E^G$ has points of accumulation on $\partial G$. Furthermore, if $E$ is the closure of a Jordan domain $H$, then $\supp\,\sigma_E^G\cap H \neq\emptyset$.
\end{theorem}

If $E=\{\zeta\}\in G$ is a singleton, then we obviously have that $-\log d_E^G(z) = g_G(z,\zeta)$, i.e. $\sigma_E^G=\delta_{\zeta}$ is the unit point mass at $\zeta.$ However, for all nontrivial compact sets $E\subset G$ the total mass $\sigma_E^G(G)$ of the representing measure is less than 1, at least when $G$ is simply connected. We conjecture that strict inequality holds in \eqref{1.5} for all domains $G$ possessing Green functions. This is different from the case of Euclidean metric in $\R^2$, see \cite{Pr1} and \cite{LP}, where the corresponding Riesz measure has unit mass. It turns out that $\sigma_E^G(G)$ depends on $G$ and $E$ in a rather complicated way. We give an explicit example of this measure below. It would be of interest to study relations between the properties of $\sigma_E^G$ and the hyperbolic geometry of $E$ in $G$. The first representation of type \eqref{1.4} for the Euclidean farthest-point distance function appeared in \cite{Bo}, where only finite sets $E$ were considered. The general Euclidean case of an arbitrary compact set $E$ was treated in \cite{Pr1}, and more detailed studies of the Riesz measure were accomplished in \cite{LP} and \cite{GN}.

Theorem \ref{thm1.1} provides important tools for proving a sharp version of \eqref{1.1} of the following form.

\begin{theorem} \label{thm1.2}
Let $E \subset G$ be a compact set, ${\rm cap}(E)>0$. Suppose
that $\nu_k, k=1,\ldots,m,$ are positive Borel measures compactly supported in $G$, such that the total mass of $\sum_{k=1}^m \nu_k$ is equal to 1. We have that
\begin{align} \label{1.6}
\sum_{k=1}^m \inf_E U_G^{\nu_k} \ge C_E^G + \sigma_E^G(G)\, \inf_E \sum_{k=1}^m U_G^{\nu_k},
\end{align}
where
\begin{equation} \label{1.7}
C_E^G := \int \inf_{z\in E} g_G(z,\zeta)\, d\mu_E^G(\zeta) - V_E^G\, \sigma_E^G(G)
\end{equation}
cannot be replaced by a smaller constant independent of $m.$
\end{theorem}

We mention another equivalent form of the best constant:
\begin{equation} \label{1.8}
C_E^G = \int \left(U_G^{\mu_E^G}(z) - V_E^G\right)\, d\sigma_E^G(z) \le 0,
\end{equation}
where the above inequality is immediate from \eqref{1.2}.

An illustration of Theorems \ref{thm1.1} and \ref{thm1.2} with explicit measure and constants is given in the following.

\begin{corollary} \label{cor1.3}
If $G=D:=\{z:|z|<1\}$ and $E=D_r:=\{z:|z|\le r\},\ 0<r<1,$ then
\begin{equation} \label{1.9}
d\sigma_{D_r}^D(z) = \frac{r(1-r^2)(1-|z|^2)}{2\pi|z|(r|z|+1)^2(r+|z|)^2}\,dxdy,\quad z=x+iy\in D,
\end{equation}
and
\begin{equation} \label{1.10}
\sigma_{D_r}^D(D) = \frac{1-r}{1+r} < 1.
\end{equation}
Furthermore, \eqref{1.6} holds with
\begin{equation} \label{1.11}
C_{D_r}^D = \log\frac{1+r^2}{2r} + \frac{1-r}{1+r}\log r.
\end{equation}
\end{corollary}

We remark that analogs of Theorems \ref{thm1.1} and \ref{thm1.2} can be proved in $\R^n,\ n\ge 3,$ see \cite{AG} and \cite{La} for the corresponding theory of Green potentials.

\section{Blaschke products} \label{sec2}

We use a well known connection of Blaschke products with Green potentials of discrete measures in the unit disk $D$. Given a finite Blaschke product of degree $n$
\[
B(z)=e^{i\theta} \prod_{j=1}^{n} \frac{z-z_j}{1-\bar z_j\, z}, \quad \{z_j\}_{j=1}^n \subset D,
\]
define the normalized counting measure in the zeros of $B$ by
\[
\nu := \frac{1}{n} \sum_{j=1}^{n} \delta_{z_j}, \quad k=1,\ldots,n,
\]
and observe that
\[
U_G^{\nu}(z) = - \frac{1}{n} \log |B(z)|, \quad k=1,\ldots,n.
\]
This simple idea yields several interesting applications of general results stated in the previous section. Let $\|\cdot\|_E$ be the supremum norm on a compact set $E\subset D$.

\begin{theorem} \label{thm2.1}
Let $E \subset D$ be a compact set of positive capacity. If $B_k,\ k=1,\ldots,m,$ are finite Blaschke products of the corresponding degrees $n_k,$ then
\begin{align} \label{2.1}
\prod_{k=1}^m \| B_k \|_E \le e^{-nC_E^D} \left\| \prod_{k=1}^m B_k \right\|_E^{\sigma_E^D(D)},
\end{align}
where $n=\sum_{k=1}^m n_k$. The constant
\begin{align} \label{2.2}
C_E^D:=\int \inf_{z\in E} \log \left| \frac{1-\bar{\zeta}z}{z-\zeta} \right| \, d \mu_E^D(\zeta) - V_E^D\,\sigma_E^D(D)
\end{align}
cannot be replaced by a smaller value independent of $m.$
\end{theorem}

The above result is a generalization of Theorem 2.1 of \cite{Pr1} on products of polynomials in uniform norms. Such inequalities for products of polynomials have been studied since at least 1930s, see \cite{Pr1}, \cite{Pr2} and \cite{PR} for history and references. However, \eqref{2.1} appears to be the first inequality of this type for Blaschke products. An example with explicit constants, for $E$ being a concentric disk, is given below.

\begin{corollary} \label{cor2.2}
If $E=D_r:=\{z:|z|\le r\},\ 0<r<1,$ in Theorem \ref{thm2.1}, then
\begin{equation} \label{2.3}
\prod_{k=1}^m \| B_k \|_{D_r} \le \left(\frac{2 r^{2r/(r+1)}}{1+r^2}\right)^n \ \left\| \prod_{k=1}^m B_k \right\|_{D_r}^{\frac{1-r}{1+r}}.
\end{equation}
\end{corollary}

The constant $C^D_E$ defined in \eqref{2.2} is asymptotically sharp, i.e., it cannot be decreased so that \eqref{2.1} still holds true for all Blaschke products as specified in Theorem \ref{thm2.1}. However, it is of interest to find an exact constant in \eqref{2.3}, which is {\em attained for each} $n\in\N$. Such a constant was found in Mahler's inequality for products of polynomials, cf. \cite{KP}. In fact, there are arrays of Blaschke products for which the constant $C_E^D$ is attained asymptotically in the sense of \eqref{2.4} below. It is a natural problem to study and characterize such extremal arrays of Blaschke products. The following two theorems address the problem by essentially stating that extremal arrays of Blaschke products are described by the equidistribution of their zeros according to the Green equilibrium measure for $E$ in $D$.

\begin{theorem} \label{thm2.3}
Let $E \subset D$ be a compact set, ${\rm cap}(E)>0$. Suppose that $B_{k,l},\ k=1,\ldots,m_l,$ are finite Blaschke products of the corresponding degrees $n_{k,l}$, and set $n_l:=\sum_{k=1}^{m_l} n_{k,l}$. Let $\tau_{n_l}$ be the normalized counting measure in the zeros of the product $\prod_{k=1}^{m_l} B_{k,l}$. If we have an array of Blaschke products $B_{k,l}$ that satisfies
\begin{align} \label{2.4}
\lim_{n_l\to\infty} \left( \frac{\prod_{k=1}^{m_l} \| B_{k,l} \|_E}{\left\| \prod_{k=1}^{m_l} B_{k,l} \right\|_E^{\sigma_E^D(D)}} \right)^{1/n_l} = e^{-C_E^D},
\end{align}
where $C_E^D$ is defined by \eqref{2.2}, then
\begin{align} \label{2.5}
\lim_{n_l\to\infty} \left\| \prod_{k=1}^{m_l} B_{k,l} \right\|_E^{1/n_l} = e^{-V_E^D}.
\end{align}
Furthermore, if \eqref{2.4} holds and one of the following conditions is satisfied:\\
(i) $E$ is a compact set with empty interior, and $D\setminus E$ is connected;\\
(ii) $E=\ov{H},$ where $H$ is a Jordan domain;\\
then the measures $\tau_{n_l}$ converge to $\mu_E^D$ in the weak* topology. \end{theorem}

We also have a partial converse for the latter theorem.

\begin{theorem} \label{thm2.4}
Let $E \subset D$ be a regular compact set. Suppose that
\[
B_{k,n}(z)=e^{i\theta_{k,n}} \frac{z-z_{k,n}}{1-{\bar z_{k,n}}z},\quad k=1,\ldots,n,\ n\in\N,
\]
are Blaschke factors. Let $\tau_n$ be the normalized counting measure in the zeros of the product $\prod_{k=1}^n B_{k,n}$. If the measures $\tau_n$ converge to $\mu_E^D$ in the weak* topology as $n\to\infty,$ then
\begin{align} \label{2.6}
\lim_{n\to\infty} \left( \frac{\prod_{k=1}^n \| B_{k,n} \|_E}{\left\| \prod_{k=1}^n B_{k,n} \right\|_E^{\sigma_E^D(D)}} \right)^{1/n} = e^{-C_E^D},
\end{align}
where $C_E^D$ is defined by \eqref{2.2}.
\end{theorem}

All results of this section may be easily transplanted into the upper half-plane $\Pi_+$, by using the standard conformal mapping $\phi(z)=(z-i)/(z+i)$ of $\Pi_+$ onto $D$. Note that finite Blaschke products on $\Pi_+$ take the following form:
\[
B(z)=e^{i\theta} \prod_{j=1}^{n} \frac{z-z_j}{z-\bar z_j}, \quad \{z_j\}_{j=1}^n \subset \Pi_+.
\]

\section{Proofs} \label{secP}

\begin{proof}[of Theorem \ref{thm1.1}]
It is clear from the definition that
\begin{align*}
-\log d_E^G(z) = \inf_{\zeta\in E} g_G(z,\zeta).
\end{align*}
Since $g_G(z,\zeta)$ is superharmonic in $G$ as a function of $z$, which holds for every $\zeta\in E,$ we have that $u(z) := \inf_{\zeta\in E} g_G(z,\zeta)$ is a superharmonic function in $G$ by Theorem 2.4.7 of \cite[p. 38]{Ra}. The Green function $g(z,\zeta)$ is positive for any $z\in G$ and any $\zeta\in E.$ Furthermore, we have for its limiting boundary values that
\[
\lim_{z\to\xi} g(z,\zeta) = 0
\]
for any $\zeta\in E$ and quasi every $\xi\in\partial G,$ see \cite{AG}, \cite{Ra} and \cite{ST}. Hence $u(z)$ is a nonnegative function in $G$ that satisfies
\[
\lim_{z\to\xi} u(z) = 0
\]
for quasi every $\xi\in\partial G.$ The generalized Maximum Principle for harmonic functions now implies that the greatest harmonic minorant of $u(z)$ in $G$ is the function $h(z)\equiv 0.$ We now conclude by Corollary 4.4.7 of \cite[p. 108]{AG} (or by Theorem $1.24^\prime$ of \cite[p. 109]{La}) that
\[
u(z)=\int g_G(z,\zeta)\, d\sigma_E^G(\zeta) = U_G^{\sigma_E^G}(z), \quad z\in G,
\]
where $\sigma_E^G$ is the Riesz representation measure supported on $G$.

Let $K\subset G$ be a regular compact set. Then we have from \eqref{1.2} (see also \cite[p. 132]{ST}) that
\[
U_G^{\mu_K^G}(z) = V_K^G, \quad z\in K.
\]
Using Fubini's theorem, we obtain that
\[
\int U_G^{\sigma_E^G}(z)\,d\mu_K^G(z) = \int U_G^{\mu_K^G}(z)\, d\sigma_E^G(z) \ge \int_K U_G^{\mu_K^G}(z)\, d\sigma_E^G(z) = V_K^G \sigma_E^G(K).
\]
On the other hand,
\[
U_G^{\sigma_E^G}(z) = u(z) = \inf_{\zeta\in E} g_G(z,\zeta) \le g_G(z,\zeta),\quad \zeta\in E,
\]
which implies by \eqref{1.2} that
\[
V_K^G \sigma_E^G(K) \le \int U_G^{\sigma_E^G}(z)\,d\mu_K^G(z) \le \int g_G(z,\zeta)\,d\mu_K^G(z) = U_G^{\mu_K^G}(\zeta) \le V_K^G.
\]
It immediately follows that
\[
\sigma_E^G(K) \le 1
\]
for any regular compact subset $K$ of $G,$ and \eqref{1.5} holds because $G$ can be exhausted by such subsets.

We now assume that $G$ is simply connected, and prove the rest of statements in this theorem. Since $E$ is not a singleton, it contains at least two distinct points $\zeta_0$ and $\zeta_1$. Consider a conformal mapping $\phi$ of $G$ onto the unit disk $D$, such that $\phi(\zeta_0)=0$ and $\phi(\zeta_1)=a\in(0,1).$ Recall that
\[
g_G(z,\zeta)=g_D(\phi(z),\phi(\zeta))=\log \frac{|1-\ov{\phi(\zeta)}\phi(z)|}{|\phi(z)-\phi(\zeta)|},\quad z,\zeta\in G.
\]
Hence $d_E^G(z) = d_{\phi(E)}^D(\phi(z)),\ z\in G,$ and all Green potentials satisfy corresponding conformal invariance property. Let $D_r:=\{w:|w|\le r\},\ 0<r<1,$ and consider $K_r:=\phi^{-1}(D_r).$ It is clear that $K_r\subset G$ is a compact set such that $E\subset K_r$ for $r$ close to 1. Recall that the Martin kernel of $D$ relative to the origin is defined by $M(w,t):=g_D(w,t)/g_D(0,t),\ w,t\in D,$ see Chapter 8 of \cite{AG}. It is known that the limiting boundary values of the Martin kernel in the unit disk coincide (up to a constant) with those of the Poisson kernel:
\[
\lim_{t\to e^{i\theta}} M(w,t) = \frac{1-|w|^2}{|e^{i\theta} - w|^2},\quad w\in D,\ \theta\in[0,2\pi),
\]
see Example 8.1.9 in \cite{AG}. In particular, we obtain that
\[
\lim_{t\to -1} M(a,t) = \lim_{t\to -1} g_D(a,t)/g_D(0,t) = \frac{1-a}{1+a} < 1,
\]
which means that there are $\epsilon,\delta>0$ such that
\[
\inf_{t\in\phi(E)} g_D(w,t) \le g_D(w,a) \le (1-\epsilon) g_D(w,0),\quad |w+1|<\delta.
\]
Using conformal invariance of Green potentials, we obtain from Example 5.13 of \cite[p. 133]{ST} that $d\mu_{K_r}^G(\phi^{-1}(re^{i\theta})) = d\theta/(2\pi)$ and $V_{K_r}^G = -\log r.$ We now let $w=re^{i\theta}=\phi(z)$, and estimate for $r$ sufficiently close to 1 that
\begin{align*}
\int \inf_{\zeta\in E} g_G(z,\zeta)\,d\mu_{K_r}^G(z) &= \dis\int_0^{2\pi} \inf_{t\in\phi(E)} g_D(w,t) \frac{d\theta}{2\pi} \le \int_{|w+1|<\delta} g_D(w,a) \frac{d\theta}{2\pi} + \int_{|w+1|\ge\delta} g_D(w,0) \frac{d\theta}{2\pi} \\ &\le (1-\epsilon) \int_{|\theta-\pi|<\delta} g_D(re^{i\theta},0) \frac{d\theta}{2\pi} + \int_{|\theta-\pi|\ge\delta} g_D(re^{i\theta},0) \frac{d\theta}{2\pi} \\ &= \left(1 - \frac{\epsilon\delta}{\pi}\right) \log\frac{1}{r} = \left(1 - \frac{\epsilon\delta}{\pi}\right) V_{K_r}^G.
\end{align*}
Arguing as in the proof of \eqref{1.5}, we obtain that
\begin{align*}
\int \inf_{\zeta\in E} g_G(z,\zeta)\,d\mu_{K_r}^G(z) &=
\int U_G^{\sigma_E^G}(z)\,d\mu_{K_r}^G(z) = \int U_G^{\mu_{K_r}^G}(z)\, d\sigma_E^G(z) \\ &\ge \int_{K_r} U_G^{\mu_{K_r}^G}(z)\, d\sigma_E^G(z) = V_{K_r}^G \sigma_E^G(K_r)
\end{align*}
 by Fubini's theorem and \eqref{1.2}. It follows that
\[
V_{K_r}^G \sigma_E^G(K_r) \le \left(1 - \frac{\epsilon\delta}{\pi}\right) V_{K_r}^G
\]
for all $r\in(0,1)$ sufficiently close to 1. Consequently, we obtain that $\sigma_E^G(G) \le \left(1 - \epsilon\delta/\pi\right)<1$ after letting $r\to 1.$

Recall our notation
\begin{align*}
u(z) = -\log d_E^G(z) = \inf_{\zeta\in E} g_G(z,\zeta) = U_G^{\sigma_E^G}(z),\quad z\in G.
\end{align*}
For any $z\in G,$ the above infimum is attained on $E$ at a point $\zeta_z$:
\[
u(z) = \inf_{\zeta\in E} g_G(z,\zeta) = g_G(z,\zeta_z),\quad z\in G,
\]
by lower semicontinuity of $g_G(z,\cdot).$ Note that $u(z)$ is harmonic in $G\setminus \supp\,\sigma_E^G,$ and let $A$ be any connected component of $G\setminus \supp\,\sigma_E^G.$ Given a fixed $z\in A$, we have that
\[
u(z) = g_G(z,\zeta_z)\quad \mbox{and}\quad u(t) \le g_G(t,\zeta_z),\quad t\in G.
\]
Hence $v(t):=u(t)-g_G(t,\zeta_z)$ is a subharmonic function that attains its maximum in $A$ at $z$, which implies that $u(t)=g_G(t,\zeta_z)$ for all $t\in A$ by the Maximum Principle. Consider $E=\ov{H},$ where $H$ is a Jordan domain. If we assume that $\supp\,\sigma_E^G \cap H = \emptyset,$ then there exists $\zeta\in\ov{H}$ such that $u(t)=g_G(t,\zeta),\ t\in H.$ But $\lim_{t\to\zeta} u(t) = \lim_{t\to\zeta} g_G(t,\zeta) = \infty,$ which means that $u(\zeta)=\infty$ by superharmonicity of $u$. This is clearly impossible, and we conclude that $\supp\,\sigma_E^G \cap H \neq \emptyset.$

If we assume that $\supp\,\sigma_E^G$ is a compact subset of $G$, then we can find $r\in(0,1)$ such that $\supp\,\sigma_E^G$ is contained in the interior of  $K_r,$ and also $E \subset K_r.$ Hence the argument from the previous paragraph gives $\zeta\in E$ such that $u(z)=g_G(z,\zeta)$ for all $z\in\partial K_r.$ We thus obtain that
\[
\int u(z)\,d\mu_{K_r}^G(z) = \int g_G(z,\zeta)\,d\mu_{K_r}^G(z) = U_G^{\mu_{K_r}^G}(\zeta) = V_{K_r}^G
\]
by \eqref{1.2} and the regularity of $\partial K_r.$ On the other hand, we have \begin{align*}
\int u(z)\,d\mu_{K_r}^G(z) = \int U_G^{\sigma_E^G}(z)\,d\mu_{K_r}^G(z) = \int U_G^{\mu_{K_r}^G}(z)\, d\sigma_E^G(z) = V_{K_r}^G \sigma_E^G(K_r) < V_{K_r}^G.
\end{align*}
This contradiction implies that $\supp\,\sigma_E^G$ must have limit points on $\partial G.$
\end{proof}

We need the following analog of Bernstein-Walsh estimate for polynomials (cf. Theorem 5.5.7 of \cite[p. 156]{Ra}).

\begin{lemma} \label{GBW}
For any unit Borel measure $\mu$ with compact support in $G$, we have
\[
U_G^{\mu}(z) - \inf_E U_G^{\mu} \ge U_G^{\mu_E^G}(z) - V_E^G, \quad z\in G.
\]
\end{lemma}

\begin{proof}
Consider the function $h(z):=U_G^{\mu}(z) - U_G^{\mu_E^G}(z),$ which is superharmonic in $G\setminus E$. Superharmonicity of $U_G^{\mu}(z)$ and \eqref{1.2} imply that
\[
\liminf_{z\to\zeta} h(z) \ge U_G^{\mu}(\zeta) - V_E^G \ge \inf_E U_G^{\mu} - V_E^G, \quad\zeta\in E.
\]
We also have for the boundary values of $h(z)$ on $\partial G$ that
\[
\lim_{z\to\zeta} h(z) = 0 \quad\mbox{for q.e. }\zeta\in\partial G
\]
by Theorem 5.1(iv) \cite[p. 124]{ST}. Recall from \eqref{1.3} that
\[
\inf_E U_G^{\mu} \le V_E^G.
\]
Hence the Minimum Principle for superharmonic functions gives that
\[
h(z) \ge \inf_E U_G^{\mu} - V_E^G \quad\mbox{for all }z\in G\setminus E.
\]
But the same inequality also holds on $E$ by \eqref{1.2}.
\end{proof}

We shall also use points minimizing discrete Green energy, which are usually called Fekete points. More precisely, an $n$-tuple of points $\{\xi_{k,n}\}_{k=1}^n\subset E$ is called the $n$th Fekete points of $E$ if
\[
\inf_{\{z_k\}_{k=1}^n\subset E} \sum_{1\le j<k\le n} g_G(z_j,z_k) = \sum_{1\le j<k\le n} g_G(\xi_{j,n},\xi_{k,n}),
\]
where the infimum is taken over all $n$-tuples $\{z_k\}_{k=1}^n\subset E.$ The following facts are well known.

\begin{lemma} \label{Fe}
Let $\{\xi_{k,n}\}_{k=1}^n$ be the $n$th Fekete points of a compact set $E\subset G.$ Then the discrete energies of Fekete points are monotonically increasing to $V_E^G$:
\begin{align} \label{5.1}
\frac{2}{n(n-1)} \sum_{1\le j<k\le n} g_G(\xi_{j,n},\xi_{k,n}) \nearrow V_E^G \quad \mbox{as } n\to\infty,
\end{align}
and
\begin{align} \label{5.2}
\lim_{n\to\infty} \inf_{z\in E} \frac{1}{n} \sum_{k=1}^n g_G(z,\xi_{k,n}) = V_E^G.
\end{align}
Furthermore, if cap$(E)>0$ then the normalized counting measures in Fekete points converge to the Green equilibrium measure in the weak* topology:
\begin{align} \label{5.3}
\tau_n:= \frac{1}{n} \sum_{k=1}^n \delta_{\xi_{k,n}} \stackrel{*}{\rightarrow} \mu_E^G \mbox{ as }n\to\infty.
\end{align}
\end{lemma}

\begin{proof}
The result of \eqref{5.1} is standard, and can be found in many books, see \cite[p. 153]{Ra} for logarithmic potentials, and see \cite[p. 160]{La} for Riesz potentials.  A very general result of this kind that covers Green potentials was obtained in \cite{Ch}.

Since we do not have precise references for \eqref{5.2} and \eqref{5.3}, we sketch their proofs here. Note that $\tau_n$ is a unit measure, so that \eqref{1.3} gives
\[
\inf_{z\in E} \frac{1}{n} \sum_{k=1}^n g_G(z,\xi_{k,n}) = \inf_{z\in E} U_G^{\tau_n}(z) \le V_E^G.
\]
On the other hand, we have for the $(n+1)$-tuple $(z,\xi_{1,n},\ldots,\xi_{n,n})\subset E$ that
\begin{align*}
\sum_{1\le j<k\le n+1} g_G(\xi_{j,n+1},\xi_{k,n+1}) \le \sum_{k=1}^n g_G(z,\xi_{k,n}) + \sum_{1\le j<k\le n} g_G(\xi_{j,n},\xi_{k,n})
\end{align*}
by the extremal property of Fekete points. Hence monotonicity in \eqref{5.1} yields that
\begin{align*}
\sum_{k=1}^n g_G(z,\xi_{k,n}) &\ge \frac{n(n+1)}{n(n+1)}\sum_{1\le j<k\le n+1} g_G(\xi_{j,n+1},\xi_{k,n+1}) - \sum_{1\le j<k\le n} g_G(\xi_{j,n},\xi_{k,n})  \\ &\ge \frac{n(n+1)}{n(n-1)}\sum_{1\le j<k\le n} g_G(\xi_{j,n},\xi_{k,n}) - \sum_{1\le j<k\le n} g_G(\xi_{j,n},\xi_{k,n}) \\
&= \frac{2}{(n-1)} \sum_{1\le j<k\le n} g_G(\xi_{j,n},\xi_{k,n}),
\end{align*}
which immediately implies that
\[
V_E^G \ge \inf_{z\in E} \frac{1}{n} \sum_{k=1}^n g_G(z,\xi_{k,n}) \ge \frac{2}{n(n-1)} \sum_{1\le j<k\le n} g_G(\xi_{j,n},\xi_{k,n}) \to V_E^G \quad \mbox{as } n\to\infty.
\]
Thus \eqref{5.2} is proved.

Since $\tau_n,\ n\in\N,$ is a sequence of positive unit Borel measures supported on $E$, we can use weak* compactness and assume that $\tau_n \stackrel{*}{\rightarrow} \tau$ as $n\to\infty$ along a subsequence $N\subset\N$. It is clear that $\tau$ is a positive unit Borel measure supported on $E$, and that $\tau_n\times\tau_n$ converges to $\tau\times\tau$ in the weak* topology along the same subsequence. Let $K_M(z,\zeta) := \min\left(g_G(z,\zeta),M\right).$ Then $K_M(z,\zeta)$ is a continuous function in $z$ and $\zeta$, and $K_M(z,\zeta)$ increases to
$g_G(z,\zeta)$ as $M\to\infty.$ Using the Monotone Convergence Theorem, we obtain for the Green energy of $\tau$ that
\begin{align*}
\iint g_G(z,\zeta)\,d\tau(z)\,d\tau(\zeta) &=
\lim_{M\to\infty} \left( \lim_{N \ni n\to\infty} \iint K_M(z,\zeta)\,
d\tau_n(z)\,d\tau_n(\zeta)\right) \\ &\le \lim_{M\to\infty} \left(
\lim_{N\ni n\to\infty} \left( \frac{2}{n^2} \sum_{1\le j<k\le n} K_M(\xi_{j,n},\xi_{k,n}) + \frac{M}{n} \right) \right) \\ &\le
\lim_{M\to\infty} \left( \liminf_{N\ni n\to\infty} \frac{2}{n^2}
\sum_{1\le j<k\le n} g_G(\xi_{j,n},\xi_{k,n}) \right) \\ &= V_E^G,
\end{align*}
where we used \eqref{5.1} on the last step. Since the Green energy of a probability measure supported on $E$ attains its minimum $V_E^G$ only for the equilibrium measure $\mu_E^G,$ see Theorem 5.10 of \cite[p. 131]{ST}, we obtain that $\tau=\mu_E^G$. The latter argument holds for any subsequence $N\in\N,$ which means that \eqref{5.3} is also proved.
\end{proof}

\begin{proof}[of Theorem \ref{thm1.2}]
Since Green potentials are superharmonic in $G$, each potential $U_G^{\nu_k}$ attains its infimum on $E$ at a point $c_k \in \partial E$:
\[
\inf_E U_G^{\nu_k} = U_G^{\nu_k}(c_k), \quad k=1,\ldots,m.
\]
Let $\nu:= \sum_{k=1}^m \nu_k,$ so that $\nu$ is a unit measure with  potential
\[
U_G^{\nu}(z) = \sum_{k=1}^m U_G^{\nu_k}(z).
\]
Using Theorem \ref{thm1.1} and Fubini's theorem, we obtain that
\begin{align*}
\sum_{k=1}^m \inf_E U_G^{\nu_k} &= \sum_{k=1}^m U_G^{\nu_k}(c_k) = \sum_{k =1}^m \int g_G(c_k,\zeta)\,d\nu_k(\zeta) \geq \int \inf_{z\in E} g_G(z,\zeta)\,d\nu(\zeta) \\
&= \int \int g_G(z,\zeta)\,d\sigma_E^G(z)\,d\nu(\zeta) = \int U_G^{\nu}(z)\, d\sigma_E^G(z).
\end{align*}
We now apply Lemma \ref{GBW} to estimate $U_G^{\nu}$ in $G$, which gives
\begin{align*}
\sum_{k=1}^m \inf_E U_G^{\nu_k} &\geq \int \left(\inf_E U_G^{\nu} + U_G^{\mu_E^G}(z) - V_E^G\right)\,d\sigma_E^G(z) \\ &=
\sigma_E^G(G)\,\inf_E U_G^{\nu} + \int \left(U_G^{\mu_E^G}(z) - V_E^G\right)\,d\sigma_E^G(z).
\end{align*}
The latter inequality proves \eqref{1.6} with constant $C_E^G$ given in \eqref{1.8}. We bring $C_E^G$ to the form of \eqref{1.7} by using
Fubini's theorem  and Theorem \ref{thm1.1} again:
\begin{align*}
\int \left(U_G^{\mu_E^G}(z) - V_E^G\right)\,d\sigma_E^G(z) &= \int U_G^{\sigma_E^G}(\zeta)\,d\mu_E^G(\zeta) - \sigma_E^G(G)\,V_E^G \\ &= \int \inf_{z\in E} g_G(z,\zeta)\, d\mu_E^G(\zeta) - V_E^G\,\sigma_E^G(G).
\end{align*}

It remains to prove the sharpness of $C_E^G$. We use the $n$th Fekete points $\{\xi_{k,n}\}_{k=1}^n\subset E$ for this purpose. Let
\[
\nu_{k,n}=\frac{1}{n}\delta_{\xi_{k,n}},
\]
so that
\[
\sum_{k=1}^n \nu_{k,n} = \tau_n
\]
is a positive unit measure supported on $E$. The left hand side of \eqref{1.6} may be written as follows:
\begin{align*}
\sum_{k=1}^n \inf_E U_G^{\nu_{k,n}} = \frac{1}{n} \sum_{k=1}^n \inf_{z\in E} g_G(z,\xi_{k,n}) = \int \inf_{z\in E} g_G(z,\zeta)\,d\tau_n(\zeta).
\end{align*}
Since the function $-\log d_E^G(\zeta) = \inf_{z\in E} g_G(z,\zeta)$ is continuous in $G$, we obtain from \eqref{5.3} that
\[
\lim_{n\to\infty} \sum_{k=1}^n \inf_E U_G^{\nu_{k,n}} = \int \inf_{z\in E} g_G(z,\zeta)\,d\mu_E^G(\zeta).
\]
On the other hand, we have for the right hand side of \eqref{1.6} that
\[
\lim_{n\to\infty} \inf_E U_G^{\tau_n} = \lim_{n\to\infty} \inf_{z\in E} \frac{1}{n} \sum_{k=1}^n g_G(z,\xi_{k,n}) = V_E^G,
\]
by \eqref{5.2}. Hence \eqref{1.6} turns into equality when we pass to the limit as $n\to\infty.$
\end{proof}

\begin{proof}[of Corollary \ref{cor1.3}]
Recall that
\[
g_D(z,\zeta)=\log\frac{|1-\bar\zeta z|}{|z-\zeta|},\quad z,\zeta\in D.
\]
Hence we have
\[
\inf_{\zeta\in D_r} g_D(z,\zeta) = \log\frac{1+r|z|}{|z|+r},\quad z\in D,
\]
because the above infimum is attained at $\zeta_m=-rz/|z|.$ Theorem $1.22^\prime$ of \cite[p. 104]{La} and Theorem \ref{thm1.1} now give that
\[
d\sigma_{D_r}^D(z) = -\frac{1}{2\pi} \Delta \left(\log\frac{1+r|z|}{|z|+r}\right)\,dxdy,\quad z=x+iy\in D.
\]
Computing Laplacian, we obtain that
\begin{align*}
d\sigma_{D_r}^D(z) &= \frac{1}{2\pi} \left(\frac{r}{|z|(|z|+r)^2} - \frac{r}{|z|(r|z|+1)^2}\right)\,dxdy \\ &= \frac{r(1-r^2)(1-|z|^2)}{2\pi|z|(r|z|+1)^2(r+|z|)^2}\,dxdy,\quad z=x+iy\in D.
\end{align*}
Elementary integration using polar coordinates implies that
\[
\sigma_{D_r}^D(D_R) = \frac{R(1-r^2)}{(rR+1)(r+R)},\quad 0\le R<1,
\]
and \eqref{1.10} follows by letting $R\to 1.$ It is known \cite[p. 133]{ST} that $V_{D_r}^D = -\log r,$ and $d\mu_E^G(z)=d\theta/(2\pi),\ z=re^{i\theta}.$ Hence we immediately obtain \eqref{1.11} from \eqref{1.7}.
\end{proof}

\begin{proof}[of Theorem \ref{thm2.1}]
Suppose that the Blaschke products $B_k$ have the following form
\[
B_k(z)=e^{i\theta_k} \prod_{j=1}^{n_k} \frac{z-z_{j,k}}{1-\bar z_{j,k}\, z}, \quad k=1,\ldots,m.
\]
We define the measures
\[
\nu_k := \frac{1}{n} \sum_{j=1}^{n_k} \delta_{z_{j,k}}, \quad k=1,\ldots,m,
\]
and observe that
\[
U_G^{\nu_k}(z) = -\frac{1}{n} \log |B_k(z)|, \quad k=1,\ldots,m.
\]
Note that the total mass of $\sum_{k=1}^m \nu_k$ is equal to 1, so that we can apply Theorem \ref{thm1.2} here. Since
\[
\inf_E U_G^{\nu_k} = -\frac{1}{n} \log \|B_k\|_E, \quad k=1,\ldots,m,
\]
\eqref{2.1} is a direct consequence of \eqref{1.6}. Using the explicit form of Green function
\[
g_D(z,\zeta)=\log\frac{|1-\bar\zeta z|}{|z-\zeta|},\quad z,\zeta\in D,
\]
we obtain \eqref{2.2} from \eqref{1.7}. In order to prove that $C_E^D$ cannot be replaced by a smaller constant, one should simply repeat the corresponding part of proof of Theorem \ref{thm1.2}.
\end{proof}

\begin{proof}[of Corollary \ref{cor2.2}]
This result is an easy combination of Theorem \ref{thm2.1} and Corollary \ref{cor1.3}.
\end{proof}

We need the following version of Lower Envelope Theorem for Green's potentials in the sequel.
\begin{lemma} \label{LowEnv}
Let $\mu_n,\ n\in\N,$ be a sequence of positive unit Borel measures that are supported in a fixed compact subset of a bounded domain $G$. If $\mu_n\stackrel{*}{\rightarrow}\mu$ then
\[
\liminf_{n\to\infty} U_G^{\mu_n}(z) = U_G^{\mu}(z)
\]
for quasi every $z\in G.$
\end{lemma}
\begin{proof} It is known that the Green potential of any measure $\nu$ supported on $G$ can be expressed via the regular logarithmic potentials of $\nu$ and its balayage $\hat\nu$ from $G$ onto $\partial G$:
\[
U_G^{\nu}(z) = U^{\nu}(z) - U^{\hat\nu}(z), \quad z\in G,
\]
see Theorem 5.1 in \cite[p. 124]{ST}. The properties of balayage immediately show that $\mu_n\stackrel{*}{\rightarrow}\mu$ implies $\hat\mu_n\stackrel{*}{\rightarrow}\hat\mu$. Since $\supp\,\hat\mu_n \subset \partial G$ and $\supp\,\hat\mu \subset \partial G$, we have that
\[
\lim_{n\to\infty} U^{\hat\mu_n}(z) = U^{\hat\mu}(z),\quad z\in G.
\]
Thus the result of this lemma follows from the Lower Envelope Theorem for logarithmic potentials (see Theorem 6.9 in \cite[p. 73]{ST}), stating that
\[
\liminf_{n\to\infty} U^{\mu_n}(z) = U^{\mu}(z)
\]
for quasi every $z\in\C.$
\end{proof}

\begin{proof}[of Theorem \ref{thm2.3}]
Let $z_{j,l},\ j=1,\ldots,n_l,$ be all zeros of the product $\prod_{k=1}^{m_l} B_k$ listed according to multiplicities. Define the individual Blaschke factors
\[
b_{j,l}(z)=\frac{z-z_{j,l}}{1-\bar z_{j,l}\, z}, \quad j=1,\ldots,n_l.
\]
We first note that \eqref{2.4} implies
\begin{align} \label{5.4}
\lim_{n_l\to\infty} \left( \frac{\prod_{j=1}^{n_l} \| b_{j,l} \|_E}{\left\| \prod_{j=1}^{n_l} b_{j,l} \right\|_E^{\sigma_E^D(D)}} \right)^{1/n_l} = e^{-C_E^D}.
\end{align}
Indeed, we have that
\[
\prod_{k=1}^{m_l} \| B_{k,l} \|_E \le \prod_{j=1}^{n_l} \| b_{j,l} \|_E \le e^{-C_E^D} \left\| \prod_{j=1}^{n_l} b_{j,l} \right\|_E^{\sigma_E^D(D)}
\]
by Theorem \ref{thm2.1}. On the other hand,
\[
\left\| \prod_{j=1}^{n_l} b_{j,l} \right\|_E = \left\| \prod_{k=1}^{m_l} B_{k,l} \right\|_E,
\]
so that \eqref{5.4} follows. Passing to a subsequence, we can assume that
\begin{align*}
\lim_{n_l\to\infty} \left\| \prod_{j=1}^{n_l} b_{j,l} \right\|_E^{1/n_l} = C
\end{align*}
and $\tau_{n_l} \stackrel{*}{\rightarrow} \tau$ as $n_l\to\infty$ hold simultaneously. It is clear that $\tau$ is a unit measure with support in $\overline{D}.$ Since $\log d_E^D(z)$ is continuous, we also have that
\begin{align*}
\lim_{n_l\to\infty} \log\left( \prod_{j=1}^{n_l} \| b_{j,l} \|_E \right)^{1/n_l} = \int \log d_E^D(z)\,d\tau(z) = -\int \inf_{\zeta\in E} g_D(z,\zeta)\,d\tau(z).
\end{align*}
Hence \eqref{5.4} gives that
\begin{align*}
\int \inf_{\zeta\in E} g_D(z,\zeta)\,d\tau(z) + \sigma_E^D(D) \log C = C_E^D.
\end{align*}
Using \eqref{1.4} and Fubini's theorem, we obtain that the left hand side of the above equation may be written as
\begin{align*}
\iint g_D(z,\zeta)\,d\sigma_E^D(\zeta)\,d\tau(z) + \sigma_E^D(D) \log C = \int \left( U_G^{\tau}(\zeta) + \log{C} \right)\,d\sigma_E^D(\zeta).
\end{align*}
Transforming the representation \eqref{2.2} for $C_E^D$ is a similar way, we have that
\begin{align} \label{5.5}
\int \left( U_G^{\tau}(\zeta) + \log{C} \right)\,d\sigma_E^D(\zeta) = \int \left( U_G^{\mu_E^D}(\zeta) - V_E^D \right)\,d\sigma_E^D(\zeta).
\end{align}
On the other hand, Lemma \ref{GBW} implies that
\begin{align*}
U_G^{\tau_{n_l}}(z) - \inf_E U_G^{\tau_{n_l}} \ge U_G^{\mu_E^D}(z) - V_E^D, \quad z\in D.
\end{align*}
Note that
\begin{align*}
- \inf_E U_G^{\tau_{n_l}} = \log \left\| \prod_{j=1}^{n_l} b_{j,l} \right\|_E^{1/n_l}.
\end{align*}
Passing to $\liminf$ as $n_l\to\infty$ and using Lemma \ref{LowEnv}, we obtain that
\begin{align} \label{5.6}
U_G^{\tau}(z) + \log C \ge U_G^{\mu_E^D}(z) - V_E^D
\end{align}
holds q.e. in $D$. The Principle of Domination (cf. Theorem 5.8 in \cite[p.130]{ST}) now implies that \eqref{5.6} holds for all $z\in D.$ Furthermore, strict inequality in \eqref{5.6} for any $z\in\supp\,\sigma_E^D$ would violate \eqref{5.5}. Hence
\begin{align} \label{5.7}
U_G^{\tau}(z) + \log C = U_G^{\mu_E^D}(z) - V_E^D, \quad z\in\supp\,\sigma_E^D.
\end{align}
Since $\supp\,\sigma_E^D$ has points of accumulation on $\partial D$ by Theorem \ref{thm1.1}, and since the limiting values of Green potentials on $\partial D$ are zero, we obtain by letting $|z|\to 1$ in \eqref{5.7} that
\[
\log C = - V_E^D.
\]
Thus \eqref{2.5} is proved, and \eqref{5.6}-\eqref{5.7} take the following form:
\begin{align} \label{5.8}
U_G^{\tau}(z) \ge U_G^{\mu_E^D}(z), \ z\in D, \quad\mbox{and}\quad U_G^{\tau}(z) = U_G^{\mu_E^D}(z), \ z\in\supp\,\sigma_E^D.
\end{align}
Let $\Omega_E$ be the connected component of $D\setminus E$ whose boundary contains the unit circumference. The function $u(z):=U_G^{\tau}(z) - U_G^{\mu_E^D}(z)$ is superharmonic in $\Omega_E,$ because $U_G^{\mu_E^D}(z)$ is harmonic in $D\setminus \partial E.$ We have that $u(z)\ge 0,\ z\in D,$ and $u(z)=0,\ z\in\supp\,\sigma_E^D,$ by \eqref{5.8}. It follows that $u$ attains its minimum in $\Omega_E$, as $\supp\,\sigma_E^D\cap\Omega_E \neq\emptyset$ by Theorem \ref{thm1.1}, so that $u(z)=0,\ z\in\Omega_E,$ by the Minimum Principle. We conclude that
\begin{align*}
U_G^{\tau}(z) = U_G^{\mu_E^D}(z), \ z\in\Omega_E,
\end{align*}
and that this equality also holds on $\partial\Omega_E$ by the continuity of potentials in the fine topology, see Corollary 5.6 in \cite[p. 61]{ST}. If $E$ satisfies conditions of Theorem \ref{thm2.3}(i), then $\partial\Omega_E=E$ and
\begin{align*}
U_G^{\tau}(z) = U_G^{\mu_E^D}(z), \ z\in D,
\end{align*}
because $D=E\cup\Omega_E$. Consequently, $\supp\,\tau \subset \supp\,\mu_E \subset \partial E.$  Integrating the above equation with respect to $\tau$, we obtain by \eqref{1.2} that
\[
\iint g_D(z,\zeta)\,d\tau(\zeta)\,d\tau(z) = \int U_G^{\tau}(z)\,d\tau(z) \le V_E^D,
\]
which means that the Green energy of $\tau$ attains the smallest possible value among all positive unit Borel measures supported on $E$. Hence $\tau=\mu_E^D$ by the uniqueness of Green equilibrium measure, see Theorems 5.10 and 5.11 in \cite[pp. 131-132]{ST}.

To prove that $\tau=\mu_E^D$ when $E=\overline{H},$ where $H$ is a Jordan domain, we follow the same argument to show that $u(z)=0,\ z\in H,$ because $u$ is a superharmonic function that attains its minimum in $H$ for $z\in\supp\,\sigma_E^D\cap H \neq\emptyset$, cf. Theorem \ref{thm1.1}. We also use here that $\supp\,\mu_E^D = \partial E = \partial\Omega_E = \partial H.$ Consequently, $U_G^{\tau}(z) = U_G^{\mu_E^D}(z)$ in $H$ and everywhere in $D=\overline{\Omega}_E \cup H.$ Using the same energy argument as before, we obtain that $\tau=\mu_E^D$ under the assumption of Theorem \ref{thm2.3}(ii).
\end{proof}

\begin{proof}[of Theorem \ref{thm2.4}]
We use \eqref{1.4} again to find that
\begin{align*}
\log \left\| B_{k,n} \right\|_E = - \inf_{\zeta\in E} g_D(z_{k,n},\zeta) = \log d_E^D(z_{k,n}), \quad k=1,\ldots,n.
\end{align*}
Since $\log d_E^D$ is continuous and $\tau_n \stackrel{*}{\rightarrow} \mu_E^D,$ we obtain that
\begin{align} \label{5.9}
\lim_{n\to\infty} \log\left( \prod_{k=1}^{n} \| B_{k,n} \|_E \right)^{1/n} = \lim_{n\to\infty} \int \log d_E^D(z)\,d\tau_n(z) = -\int \inf_{\zeta\in E} g_D(z,\zeta)\,d\mu_E^D(z).
\end{align}
On the other hand, we have that
\begin{align*}
\log \left\| \prod_{k=1}^{n} B_{k,n} \right\|_E ^{1/n} =  - \inf_E U_G^{\tau_n}.
\end{align*}
Using lower semicontinuity of $U_G^{\tau_n}$, we conclude that the infimum is attained at a point $z_n\in E.$ We can assume that $\lim_{n\to\infty} z_n = z_0 \in E,$ after passing to a subsequence. The Principle of Descent (see Theorem 6.8 of \cite[p. 70]{ST} and Theorem 1.3 of \cite[p. 62]{La}) implies that
\begin{align*}
\liminf_{n\to\infty} U_G^{\tau_n}(z_n) \ge U_G^{\mu_E^D}(z_0) = V_E^D,
\end{align*}
where the last equality holds by \eqref{1.2} and regularity of $E.$ Hence
\begin{align*}
\limsup_{n\to\infty} \log \left\| \prod_{k=1}^{n} B_{k,n} \right\|_E ^{1/n} \le - V_E^D,
\end{align*}
and we obtain from \eqref{5.9} that
\begin{align*}
\liminf_{n\to\infty} \left( \frac{\prod_{k=1}^n \| B_{k,n} \|_E}{\left\| \prod_{k=1}^n B_{k,n} \right\|_E^{\sigma_E^D(D)}} \right)^{1/n} \ge e^{-C_E^D}.
\end{align*}
But
\begin{align*}
\limsup_{n\to\infty} \left( \frac{\prod_{k=1}^n \| B_{k,n} \|_E}{\left\| \prod_{k=1}^n B_{k,n} \right\|_E^{\sigma_E^D(D)}} \right)^{1/n} \le e^{-C_E^D}
\end{align*}
by \eqref{2.1}, so that \eqref{2.6} follows.
\end{proof}

\affiliationone{
   Department of Mathematics\\
   Oklahoma State University\\
   Stillwater, OK 74078\\
   U.S.A.
\email{igor@math.okstate.edu}}

\end{document}